\documentclass[11 pt]{amsart}
\usepackage{latexsym,amscd,amssymb, graphicx, amsthm, mathtools}
\usepackage{tikz}

\usepackage[margin=1in]{geometry}

\numberwithin{equation}{section}

\newtheorem{theorem}{Theorem}[section]

\newtheorem{lemma}[theorem]{Lemma}
\newtheorem{conjecture}[theorem]{Conjecture}

\newtheorem{remark}[theorem]{Remark}

\theoremstyle{definition}

\newcommand{\CC}{\mathbb{C}}

\newcommand{\Frob}{\mathrm{Frob}}

\newcommand{\Res}{\mathrm{Res}}
\newcommand{\Hom}{\mathrm{Hom}}
\newcommand{\triv}{\mathbf{1}}

\newcommand{\grFrob}{\mathrm{grFrob}}

\newcommand{\xx}{\mathbf {x}}
\newcommand{\yy}{\mathbf {y}}
\newcommand{\ttheta}{\boldsymbol \theta}
\newcommand{\xxi}{\boldsymbol \xi}

\newcommand{\symm}{\mathfrak{S}}

\begin{document}

\title[The fermionic theta coinvariant conjecture]
{A proof of the fermionic theta coinvariant conjecture}

\author{Alessandro Iraci, Brendon Rhoades, and Marino Romero}
\address
{D\'epartement de Math\'ematiques \newline \indent
Universit\'e du Qu\'ebec \`a Montr\'eal \newline \indent
Montr\'eal, QC, H2X 3Y7, Canada \newline \indent
\texttt{\textup{iraci.alessandro@uqam.ca}} \newline \newline \indent
Department of Mathematics \newline \indent
University of California, San Diego \newline \indent
La Jolla, CA, 92093-0112, USA \newline \indent
\texttt{\textup{bprhoades@ucsd.edu}} \newline \newline \indent
Department of Mathematics \newline \indent
University of Pennsylvania \newline \indent
Philadelphia, PA, 19104-6395, USA \newline \indent
\texttt{\textup{mar007@sas.upenn.edu}}}

\begin{abstract}
Let $(x_1, \dots, x_n, y_1, \dots, y_n)$ be a list of $2n$ commuting variables, 
$(\theta_1, \dots, \theta_n, \xi_1, \dots, \xi_n)$ be a list of $2n$ anticommuting variables, and
$\CC[\xx_n, \yy_n] \otimes \wedge \{\ttheta_n, \xxi_n\}$ be the algebra generated by these variables.
D'Adderio, Iraci, and Vanden Wyngaerd introduced the {\em Theta operators} on the ring of symmetric functions
and used them to conjecture a formula for the quadruply-graded $\symm_n$-isomorphism type of 
$\CC[\xx_n,\yy_n] \otimes \wedge \{\ttheta_n, \xxi_n\}/I$ where $I$ is the ideal generated by $\symm_n$-invariants with vanishing 
constant term.
We prove their conjecture in the `purely fermionic setting' obtained by setting the commuting variables equal $x_i, y_i$ equal to zero.
\end{abstract}

\maketitle

\section{Introduction}
\label{Introduction}

The {\em diagonal coinvariant ring} $DR_n$ is obtained from the rank $2n$ polynomial ring 
$\CC[\xx_n,\yy_n] = \CC[x_1, \dots, x_n, y_1, \dots, y_n]$ by factoring out the ideal generated by $\symm_n$-invariants 
with vanishing constant term. 
The ring $DR_n$ is a bigraded $\symm_n$-module;
Haiman used algebraic geometry to calculate its isomorphism type \cite{Haiman}. 
In recent years, researchers in algebraic combinatorics studied variants of $DR_n$ involving mixtures of commuting 
and anticommuting variables 
\cite{BRT, Bergeron, DIV, HS, KRNC, KRLefschetz, RWVan, RWSuperpartition, SW, SW2, ZabrockiModule, ZabrockiBlog}.
Drawing terminology from supersymmetry, we will refer to commuting variables as {\em bosonic}
and anticommuting variables as {\em fermionic}.  
D'Adderio, Iraci, and Vanden Wyngaerd conjectured \cite{DIV}
a generalization of Haiman's result involving two sets of bosonic and two sets of 
fermionic variables of which Haiman's result forms the `purely bosonic case'. We prove the `purely fermionic case' of their conjecture.

We begin by fixing some notation.
Let $\CC[\xx_n, \yy_n] \otimes \wedge \{\ttheta_n, \xxi_n\}$ be the tensor product
\begin{equation*}
\CC[x_1, \dots, x_n, y_1, \dots, y_n] \otimes \wedge \{ \theta_1, \dots, \theta_n, \xi_1, \dots, \xi_n \}
\end{equation*}
of a rank $2n$ symmetric algebra with a rank $2n$ exterior algebra. 
This ring carries four independent gradings (two bosonic and two fermionic) and
the diagonal action of the symmetric group $\symm_n$ 
\begin{equation*}
    w \cdot x_i = x_{w(i)} \quad w \cdot y_i = y_{w(i)} \quad w \cdot \theta_i = \theta_{w(i)} \quad 
    w \cdot \xi_i = \xi_{w(i)}
\end{equation*}
preserves this quadrigrading. Writing $I_n$ for the ideal generated by $\symm_n$-invariants with vanishing constant term,
the quotient 
\begin{equation}
    TDR_n \coloneqq (\CC[x_1, \dots, x_n, y_1, \dots, y_n] \otimes \wedge \{ \theta_1, \dots, \theta_n, \xi_1, \dots, \xi_n \})/I_n
\end{equation}
is a quadruply-graded $\symm_n$-module. The notation $TDR_n$ alludes to its status as a `supersymmetric double' of $DR_n$ with
$T$wice as many 
generators.

Let $\Lambda$ denote the ring of symmetric functions in $\xx = (x_1, x_2, \dots )$
over the ground field $\CC(q,t)$.
Given any element $F \in \Lambda$, D'Adderio, Iraci, and 
Vanden Wyngaerd introduced \cite{DIV} a {\em Theta operator}
$\Theta_F: \Lambda \rightarrow \Lambda$ whose definition is recalled in Section~\ref{Background}.
Building on the techniques in \cite{DIV},
D'Adderio and Mellit  \cite{DM} used Theta operators 
to prove the rise version of the {\em Delta Conjecture} 
of Haglund, Remmel, and Wilson \cite{HRW}.
D'Adderio and Romero \cite{DR} rederived and clarified a slew of symmetric function identities using Theta operators,
drastically shortening many of their proofs.

In this paper, we will only consider Theta operators indexed by elementary symmetric functions $e_i \in \Lambda$; we abbreviate
$\Theta_i \coloneqq \Theta_{e_i}$.
If $G \in \Lambda$ has degree $d$ in the $\xx$-variables, then $\Theta_i \, G$ has degree $d+i$ and will typically involve the parameters
$q$ and $t$ (even if $G$ itself does not).
D'Adderio, Iraci, and Vanden Wyngaerd 
conjectured \cite{DIV} that symmetric functions of the form $\Theta_i \Theta_j \nabla e_{n-i-j}$ 
(where $\nabla$ is the Bergeron-Garsia nabla operator) 
determine the quadruply-graded $\symm_n$-structure of $TDR_n$. Deferring various definitions to 
Section~\ref{Background}, their conjecture may be stated as follows.

\begin{conjecture}
\label{coinvariant-conjecture}
{\em (D'Adderio-Iraci-Vanden Wyngaerd \cite{DIV})}
Let $(TDR_n)_{i,j}$ be the piece of $TDR_n$ with homogeneous $\theta$-degree $i$ and $\xi$-degree $j$.
The space $(TDR_n)_{i,j}$ vanishes whenever $i + j \geq n$. If  $i + j < n$ we have
\begin{equation*}
    \grFrob((TDR_n)_{i,j};q,t) = \Theta_i \Theta_j \nabla e_{n-i-j}
\end{equation*}
where $q$ tracks $x$-degree and $t$ tracks $y$-degree.
\end{conjecture}
The case $i = j = 0$ of Conjecture~\ref{coinvariant-conjecture} amounts to setting the fermionic $\theta$ and $\xi$-variables to zero
and is equivalent to Haiman's  two-bosonic result $\grFrob(DR_n;q,t) = \nabla e_n$.
Setting the $y$-variables and $\xi$-variables to zero, Conjecture~\ref{coinvariant-conjecture} reduces to 
the one-bosonic, one-fermionic 
`superspace coinvariant conjecture' of the Fields Institute combinatorics group
(see \cite{ZabrockiModule,ZabrockiBlog}).
If only the $\xi$-variables are set to zero, Conjecture~\ref{coinvariant-conjecture} yields
a two-bosonic, one-fermionic conjecture of Zabrocki \cite{ZabrockiModule}
tied to the {\em Delta operators} on $\Lambda$.

In this paper we give additional evidence for Conjecture~\ref{coinvariant-conjecture}
by proving its purely fermionic case. 
The {\em fermionic diagonal coinvariant ring} 
\begin{equation}
    FDR_n \coloneqq \wedge \{ \ttheta_n, \xxi_n \}/\langle   \wedge \{\ttheta_n, \xxi_n\}^{\symm_n}_+ \rangle
\end{equation}
is obtained from the rank $2n$ exterior algebra $\wedge \{ \ttheta_n, \xxi_n \}$ by modding out by the ideal 
generated by $\symm_n$-invariants with vanishing constant term.
Equivalently, the ring $FDR_n$ is obtained from $TDR_n$ by setting the $x$-variables and $y$-variables equal to zero.
The ring $FDR_n$ is a bigraded $\symm_n$-module.
Jongwon Kim and Rhoades introduced $FDR_n$ in \cite{KRLefschetz}; Jesse Kim and Rhoades used $FDR_n$ as a model
for resolving a set partition of $\{1, \dots, n\}$ into a linear combination of noncrossing set partitions 
\cite{KRNC}.
Our main result is as follows.

\begin{theorem}
\label{main-theorem}
We have $(FDR_n)_{i,j} = 0$ whenever $i + j \geq n$. When $i + j < n$ we have  
\begin{equation*}
    \Frob \, (FDR_n)_{i,j} = \Theta_i \Theta_j \nabla e_{n-i-j}  \mid_{q = t = 0}.
\end{equation*}
\end{theorem}

\begin{remark}
\label{surjection-remark}
The canonical surjection $DR_n \otimes FDR_n \twoheadrightarrow TDR_n$ applies to show that the 
bidegree support assertion in Theorem~\ref{main-theorem} implies the corresponding assertion in
Conjecture~\ref{coinvariant-conjecture}. 
\end{remark}

To prove Theorem~\ref{main-theorem}, we apply a result of 
Jongwon Kim and Rhoades \cite{KRLefschetz} which expresses 
$\Frob \, (FDR_n)_{i,j}$ in terms of Kronecker products of hook-shaped
Schur functions. 
To obtain a recursive structure on the characters 
$\Frob \, (FDR_n)_{i,j},$ we prove a result 
(Lemma~\ref{schur-kronecker-recursion}) on applying the skewing operators
$h_{d}^{\perp}$ to Kronecker products
$s_{\lambda} \ast s_{\mu}$ which may be of independent interest
for studying Kronecker
products in general. For the right-hand side 
$ \Theta_i \Theta_j \nabla e_{n-i-j}  \mid_{q = t = 0}$
of Theorem~\ref{main-theorem},
we apply a
result of D'Adderio and Romero \cite{DR} for applying the operators
$h_d^{\perp}$ to expressions of the form
$\Theta_i \Theta_j  \widetilde{H}_{(n-i-j)}(\xx;q,t)$  and study what happens
at the evaluation $q,t \rightarrow 0$.
We check that the recursions
coincide, and Theorem~\ref{main-theorem} is proven.

The rest of the paper is organized as follows. 
In Section~\ref{Background} we give the required background on symmetric functions
and representation theory. 
In Section~\ref{Proof} we prove Theorem~\ref{main-theorem}.

\section{Background}
\label{Background}
We give background material on symmetric functions and the representation
theory of symmetric groups. We use the operation $F[G]$ of {\em plethysm}
throughout. For a more thorough exposition of this material, we refer the reader to
\cite{HaglundBook}.

As in the introduction, we write $\Lambda = \bigoplus_{n \geq 0} \Lambda_n$ for the 
graded ring of symmetric functions in the variable set 
$\xx = (x_1, x_2, \dots )$ over the ground field $\CC(q,t)$. 
For $n \geq 0$, we write $e_n = e_n(\xx), h_n = h_n(\xx) \in \Lambda_n$ for the {\em elementary}
and {\em complete homogeneous} symmetric functions of degree $n$.

Bases of $\Lambda_n$ are indexed by partitions $\lambda \vdash n$.
The elementary basis $e_{\lambda}$ and homogeneous basis $h_{\lambda}$ are defined by setting
$e_{\lambda} \coloneqq e_{\lambda_1} e_{\lambda_2} \cdots$ and $h_{\lambda} \coloneqq h_{\lambda_1} h_{\lambda_2} \cdots$.
We will also use the basis of {\em Schur functions} $s_{\lambda} = s_{\lambda}(\xx)$ and the basis of 
{\em (modified) Macdonald polynomials} $\widetilde{H}_{\lambda}(\xx;q,t)$.  

The {\em Hall inner product} $\langle -, - \rangle$ on $\Lambda$ is defined by declaring the Schur basis to be orthogonal:
\begin{equation*}
    \langle s_{\lambda}, s_{\mu} \rangle = \delta_{\lambda, \mu}
\end{equation*}
where $\delta_{\lambda,\mu}$ is the Kronecker delta.
Given $F \in \Lambda$, write $F^{\bullet}: \Lambda \rightarrow \Lambda$ for the operator $F^{\bullet}(G) \coloneqq F G$ of multiplication
by $F$. We write $F^{\perp}: \Lambda \rightarrow \Lambda$ for the adjoint of the operator $F^{\bullet}$; it is characterized by 
\begin{equation}
    \langle F^{\perp} G, H \rangle = \langle G, F^{\bullet} H \rangle
\end{equation}
for all $G, H \in \Lambda$. The application of $F^{\perp}$
to a symmetric function is often referred to as
{\em skewing} with respect to $F$.
  
Eigenoperators on the Macdonald basis have proven to be remarkable objects in symmetric function theory.
We use two such Macdonald eigenoperators in this paper. The first is the {\em nabla operator}
$\nabla: \Lambda \rightarrow \Lambda$ defined by
\begin{equation}
    \nabla: \widetilde{H}_{\mu}(\xx;q,t) \mapsto \prod_{(i,j) \in \mu} q^{i-1} t^{j-1} \cdot \widetilde{H}_{\mu}(\xx;q,t)
\end{equation}
where the product is over all cells $(i,j)$ in the Young diagram of $\mu$. 
Similarly, the operator $\Pi: \Lambda \rightarrow \Lambda$ is given by
\begin{equation}
    \Pi: \widetilde{H}_{\mu}(\xx;q,t) \mapsto \prod_{\substack{(i,j) \in \mu \\ (i,j) \neq (1,1)}} (1 - q^{i-1} t^{j-1}) \cdot \widetilde{H}_{\mu}(\xx;q,t)
\end{equation}
where the product is over all cells $(i,j) \neq (1,1)$ in the Young diagram of $\mu$.  We abbreviate the eigenvalue of $\widetilde{H}_{\mu}(\xx;q,t)$ under 
the operator $\Pi$ as
\begin{equation}
  \Pi_{\mu} :=  \prod_{\substack{(i,j) \in \mu \\ (i,j) \neq (1,1)}}
  (1-q^{i-1}t^{j-1}). 
\end{equation}
The omission of $(1,1)$ in this product
assures that the operator $\Pi$ is nonzero and, in fact, invertible.

We are ready to define the Theta operators of \cite{DIV}. 
Given $F = F(\xx) \in \Lambda$, let $F \left[  \frac{\xx}{M} \right]$
be the symmetric function obtained by plethystically evaluating $F$ at $\frac{\xx}{M}$. 
where $M = (1-q)(1-t)$.
The {\em Theta operator} $\Theta_F: \Lambda \rightarrow \Lambda$ is obtained by conjugating the multiplication operator
$F \left[ \frac{\xx}{M} \right]^{\bullet}$ by $\Pi$. That is, we set
\begin{equation}
    \Theta_F \coloneqq  \Pi \circ F\left[ \frac{\xx}{M} \right]^{\bullet} \circ \Pi^{-1}.
\end{equation}
Assuming $F$ is homogeneous,
the operator $\Theta_F$ is homogeneous of degree 
$\mathrm{deg}(F)$ on the graded ring $\Lambda = \bigoplus_{n \geq 0} \Lambda_n$. 
As explained in the introduction, we will only use Theta operators indexed by elementary symmetric functions,
and so abbreviate $\Theta_d \coloneqq \Theta_{e_d}$.

We recall some basic ideas from group representation theory. 
If $G$ is a group and $V_1, V_2$ are $G$-modules, we write $V_1 \otimes V_2$
for their {\em Kronecker product} (or {\em internal product}). This is the 
$G$-module with underlying vector space given by the tensor product of 
$V_1$ and $V_2$ with $G$-module structure 
$g \cdot (v_1 \otimes v_2) := (g \cdot v_1) \otimes (g \cdot v_2)$.
For us, the group $G$ will either be a symmetric group $\symm_n$ or a 
parabolic subgroup $\symm_j \times \symm_{n-j}$ thereof.

If $G$ and $H$ are groups, $V$ is a $G$-module, and $W$ is an $H$-module,
we write $V \boxtimes W$ for the $(G \times H)$-module whose 
underlying vector space is the tensor product of $V$ and $W$ and whose
module structure is determined by 
$(g,h) \cdot (v \otimes w) := (g \cdot v) \otimes (h \cdot w)$.
For us, both $G$ and $H$ will be symmetric groups.
We use the distinct notations $\otimes$ and $\boxtimes$
to avoid confusion in the proof of Lemma~\ref{schur-kronecker-recursion} below.

Irreducible representations of the symmetric group $\symm_n$ over $\CC$
are in bijective correspondence with partitions of $n$.
Given $\lambda \vdash n$, we write $S^{\lambda}$ for the corresponding $\symm_n$-irreducible. 
If $V$ is any finite-dimensional $\symm_n$-module, there are unique multiplicities $c_{\lambda} \geq 0$ such that 
$V \cong \bigoplus_{\lambda \vdash n} c_{\lambda} S^{\lambda}$. The {\em Frobenius image} $\Frob(V) \in \Lambda_n$ is the symmetric function
\begin{equation}
    \Frob(V) \coloneqq \sum_{\lambda \vdash n} c_{\lambda} \cdot s_{\lambda}
\end{equation}
obtained by replacing each irreducible factor with the corresponding Schur function. 
More generally, if $V = \bigoplus_{i \geq 0} V_i$ is a graded $\symm_n$-module with each $V_i$ finite-dimensional, its 
{\em graded Frobenius image} is 
\begin{equation}
    \grFrob(V;q) = \sum_{i \geq 0} \Frob(V_i) \cdot q^i
\end{equation}
and if $V = \bigoplus_{i,j \geq 0} V_{i,j}$ is a bigraded $\symm_n$ module we have the {\em bigraded Frobenius image}
\begin{equation}
    \grFrob(V;q,t) = \sum_{i,j \geq 0} \Frob(V_{i,j}) \cdot q^i t^j.
\end{equation}

 Operations on symmetric functions correspond to operations on symmetric group modules via the Frobenius map. 
 For example, if $V$ is an $\symm_n$-module and $W$ is an $\symm_m$-module, 
 the {\em induction product} of $V$ and $W$ 
 is $V \circ W \coloneqq \mathrm{Ind}_{\symm_n \times \symm_m}^{\symm_{n+m}}(V \boxtimes W)$ 
 where the embedding $\symm_n \times \symm_m \subset \symm_{n+m}$ is obtained by letting $\symm_n$ permute the first $n$ 
 letters and $\symm_m$ permute the last $m$ letters.
 We have
 \begin{equation}
     \Frob(V \circ W) = \Frob(V) \cdot \Frob(W).
 \end{equation}
 Defined for partitions $\lambda \vdash n, \mu \vdash m$, and $\nu \vdash n+m$, the {\em Littlewood-Richardson coefficients}
 $c_{\lambda,\mu}^{\nu}$ are the structure coefficients for this product in the Schur basis. They are characterized by either of the formulas
 \begin{equation}
    S^{\lambda} \circ S^{\mu} \cong \bigoplus_{\nu \vdash m+n} c_{\lambda,\mu}^{\nu} S^{\nu} \quad \quad \text{or} \quad \quad
    s_{\lambda} \cdot s_{\mu} = \sum_{\nu \vdash m+n} c_{\lambda,\mu}^{\nu} \cdot s_{\nu}.
 \end{equation}
 The {\em Littlewood-Richardson rule} gives a combinatorial interpretation of the nonnegative integers $c_{\lambda,\mu}^{\nu}.$

 As another example, the {\em Kronecker product} on the space of degree $n$
 symmetric functions $\Lambda_n$ is the bilinear operation $\ast$ characterized by
 \begin{equation}
     \Frob(S^{\lambda} \otimes S^{\mu}) = s_{\lambda} \ast s_{\mu}  
 \end{equation}
 for all $\lambda, \mu \vdash n$.
 The nonnegative integers $g_{\lambda,\mu,\nu}$ indexed by triples of partitions
 $\lambda, \mu, \nu \vdash n$ determined by 
 $s_{\lambda} \ast s_{\mu} = \sum_{\nu \vdash n} g_{\lambda,\mu,\nu} \cdot s_{\nu}$
 are the {\em Kronecker coefficients}.
 Finding a combinatorial rule for $g_{\lambda,\mu,\nu}$ is a famous
 open problem.

\section{Proof of Theorem~\ref{main-theorem}}
\label{Proof}

Theorem~\ref{main-theorem} asserts an equality of symmetric functions.
Our strategy for proving this equality is to show that both sides satisfy the following recursion.

\begin{lemma}
\label{h-perp-recursion}
Let $F, G \in \Lambda$ be symmetric functions with vanishing constant terms. Suppose that 
$h_j^{\perp} F = h_j^{\perp} G$ for all $j \geq 1$. Then $F = G$.
\end{lemma}

\begin{proof}
For any partition $\lambda = (\lambda_1, \lambda_2, \lambda_3, \dots )$ with $\lambda_1 > 0$, we have
\begin{equation}
    \langle F, h_{\lambda} \rangle =
    \langle h_{\lambda_1}^{\perp} F, h_{\bar{\lambda}} \rangle =
    \langle h_{\lambda_1}^{\perp} G, h_{\bar{\lambda}} \rangle =
    \langle G, h_{\lambda} \rangle
\end{equation}
where $\bar{\lambda} = (\lambda_2, \lambda_3, \dots )$ and the result follows since the $h_{\lambda}$ form a basis 
of $\Lambda$ and $\langle -, - \rangle$ is an inner product.
\end{proof}

We handle the representation theoretic side $\Frob \, (FDR_n)_{i,j}$ of Theorem~\ref{main-theorem} first.
Our starting point is the following result of Jongwon Kim and Rhoades \cite{KRLefschetz} 
which describes this symmetric function
in terms of Kronecker products of hook-shaped Schur functions.

\begin{theorem}
\label{fdr-isomorphism-type}
{\em (Jongwon Kim-Rhoades \cite{KRLefschetz})}
We have $(FDR_n)_{i,j} = 0$ whenever $i+j \geq n$. When $i+j < n$ we have 
\begin{equation*}
    \Frob \, (FDR_n)_{i,j} = s_{(n-i,1^i)} \ast s_{(n-j,1^j)} - s_{(n-i+1,1^{i-1})} \ast s_{(n-j+1,1^{j-1})}
\end{equation*}
where by convention $s_{(n-i+1,1^{i-1})} \ast s_{(n-j+1,1^{j-1})} = 0$ when $i = 0$ or $j = 0$.
\end{theorem}

Theorem~\ref{fdr-isomorphism-type} was proven by showing that the $\symm_n$-invariant element 
$\theta_1 \xi_1 + \cdots + \theta_n \xi_n \in \wedge \{ \ttheta_n, \xxi_n \}$ satisfies a kind of `bigraded Lefschetz property'.
While there exist expressions for the Schur expansion of $s_{\lambda} \ast s_{\mu}$ when $\lambda, \mu \vdash n$
are hooks (see e.g. \cite{Rosas}), these formulas are rather complicated.
With an eye towards Lemma~\ref{h-perp-recursion}, we give a recursive rule for applying $h_j^{\perp}$ to an arbitrary Kronecker product 
of Schur functions.

\begin{lemma}
\label{schur-kronecker-recursion}
Let $1 \leq j \leq n$ and let $\lambda^{(1)}, \lambda^{(2)} \vdash n$ be two partitions.
We have 
\begin{equation}
\label{hk-perp-equation}
    h_j^{\perp}(s_{\lambda^{(1)}} \ast s_{\lambda^{(2)}}) =
    \sum_{\substack{\mu \vdash j \\ \nu^{(1)}, \nu^{(2)} \vdash n-j}}
    c_{\mu,\nu^{(1)}}^{\lambda^{(1)}} \cdot c_{\mu,\nu^{(2)}}^{\lambda^{(2)}}
    (s_{\nu^{(1)}} \ast s_{\nu^{(2)}})
\end{equation}
where the $c_{\mu,\nu^{(i)}}^{\lambda^{(i)}}$ are Littlewood-Richardson coefficients.
\end{lemma}

The proof of Lemma~\ref{schur-kronecker-recursion} requires 
both of the module operations $\boxtimes$ and $\otimes$ introduced
in Section~\ref{Background}.

\begin{proof}
It is well-known that, for any $\symm_n$-module $V$,
the degree $n-k$ symmetric function $h_j^{\perp} \Frob \, V$ has the algebraic interpretation
\begin{equation}
    h_j^{\perp} \Frob \, V = 
    \Frob \, \Hom_{\symm_j}(\triv_{\symm_j}, \Res^{\symm_n}_{\symm_j \times \symm_{n-j}} \, V)
\end{equation}
where the $\Hom_{\symm_j}$-space is an $\symm_{n-j}$-module by means of the
second factor of 
$\symm_j \times \symm_{n-j}$. In our situation, this reads
\begin{align}
    h_j^{\perp} & (s_{\lambda^{(1)}} \ast s_{\lambda^{(2)}}) =
    \Frob \, \Hom_{\symm_j}\left(\triv_{\symm_j}, \Res^{\symm_n}_{\symm_j \times \symm_{n-j}} \,
    (S^{\lambda^{(1)}} \otimes S^{\lambda^{(2)}}) \right) \\
    &=
    \Frob \, \Hom_{\symm_j}\left(\triv_{\symm_j}, (\Res^{\symm_n}_{\symm_j \times \symm_{n-j}} \,
    S^{\lambda^{(1)}}) \otimes
    (\Res^{\symm_n}_{\symm_j \times \symm_{n-j}} S^{\lambda^{(2)}}) \right) \\
    &= 
    \Frob \, \Hom_{\symm_j}\left(\triv_{\symm_j}, 
    \bigoplus_{\substack{\mu^{(1)}, \mu^{(2)} \vdash j \\ \nu^{(1)}, \nu^{(2)} \vdash n-j}} \,
    c_{\mu^{(1)}, \nu^{(1)}}^{\lambda^{(1)}} \cdot 
    c_{\mu^{(2)}, \nu^{(2)}}^{\lambda^{(2)}} \cdot
    (S^{\mu^{(1)}} \boxtimes S^{\nu^{(1)}}) \otimes (S^{\mu^{(2)}} \boxtimes S^{\nu^{(2)}})\right) \\
        &= 
    \Frob \, \Hom_{\symm_j}\left(\triv_{\symm_j}, 
    \bigoplus_{\substack{\mu^{(1)}, \mu^{(2)} \vdash j \\ \nu^{(1)}, \nu^{(2)} \vdash n-j}} \,
    c_{\mu^{(1)}, \nu^{(1)}}^{\lambda^{(1)}} \cdot 
    c_{\mu^{(2)}, \nu^{(2)}}^{\lambda^{(2)}} \cdot
    (S^{\mu^{(1)}} \otimes S^{\mu^{(2)}}) \boxtimes (S^{\nu^{(1)}} \otimes S^{\nu^{(2)}})\right) \\
    &= \sum_{\substack{\mu^{(1)}, \mu^{(2)} \vdash j \\ \nu^{(1)}, \nu^{(2)} \vdash n-j}}
    c_{\mu^{(1)}, \nu^{(1)}}^{\lambda^{(1)}} \cdot 
    c_{\mu^{(2)}, \nu^{(2)}}^{\lambda^{(2)}} \cdot
    \dim \, \Hom_{\symm_j}(\triv_{\symm_j}, S^{\mu^{(1)}} \otimes S^{\mu^{(2)}}) \cdot
    s_{\nu^{(1)}} \ast s_{\nu^{(2)}}
\end{align}
where we used the fact that restriction functors commute with Kronecker products and the 
consequence
\begin{equation}
 \Res^{\symm_n}_{\symm_j \times \symm_{n-j}} \, S^{\lambda} \cong
 \bigoplus_{\substack{\mu \vdash j \\ \nu \vdash n-j}} c_{\mu,\nu}^{\lambda} (S^{\mu} \boxtimes S^{\nu}) \quad
 \quad (\lambda \vdash n)
\end{equation}
of Frobenius reciprocity. 

The multiplicities 
$g_{\rho,\mu^{(1)},\mu^{(2)}} = \dim(\Hom(S^{\rho}, S^{\mu^{(1)}} \otimes S^{\mu^{(2)}}))$
of Schur functions $s_{\rho}$ in general Kronecker products $s_{\mu^{(1)}} \ast s_{\mu^{(2)}}$ are difficult to compute.
However, when $\rho = (j)$ and $S^{\rho} = \triv_{\symm_j}$ as in our setting, character orthogonality gives
\begin{equation}
    \dim(\Hom_{\symm_j}(\triv_{\symm_j}, S^{\mu^{(1)}} \otimes S^{\mu^{(2)}})) = \begin{cases}
    1 & \mu^{(1)} = \mu^{(2)}, \\
    0 & \text{otherwise}.
    \end{cases}
\end{equation}
Adding this information to the above string of equalities gives
\begin{align}
    h_j^{\perp}&(s_{\lambda^{(1)}} \ast s_{\lambda^{(2)}}) \\ &=
 \sum_{\substack{\mu^{(1)}, \mu^{(2)} \vdash j \\ \nu^{(1)}, \nu^{(2)} \vdash n-j}}
    c_{\mu^{(1)}, \nu^{(1)}}^{\lambda^{(1)}} \cdot 
    c_{\mu^{(2)}, \nu^{(2)}}^{\lambda^{(2)}} \cdot
    \dim \, \Hom_{\symm_j}(\triv_{\symm_j}, S^{\mu^{(1)}} \otimes S^{\mu^{(2)}}) \cdot
    s_{\nu^{(1)}} \ast s_{\nu^{(2)}} \\
    &= \sum_{\substack{\mu \vdash j \\ \nu^{(1)}, \nu^{(2)} \vdash n-j}}
    c_{\mu, \nu^{(1)}}^{\lambda^{(1)}} \cdot 
    c_{\mu, \nu^{(2)}}^{\lambda^{(2)}} \cdot
    s_{\nu^{(1)}} \ast s_{\nu^{(2)}}
\end{align}
and our proof is complete.
\end{proof}

Our ability to put Lemma~\ref{schur-kronecker-recursion} to good use is bounded by our understanding
of the products 
$c_{\mu,\nu^{(1)}}^{\lambda^{(1)}} \cdot c_{\mu,\nu^{(2)}}^{\lambda^{(2)}}$ of 
Littlewood-Richardson coefficients
appearing therein.
Thanks to Theorem~\ref{fdr-isomorphism-type}, for our purposes both $\lambda^{(1)}$ and $\lambda^{(2)}$
will be hook-shaped partitions so that these coefficients will be fairly simple.
The situation becomes more complicated for general $\lambda^{(1)}, \lambda^{(2)} \vdash n$,
but Lemma~\ref{schur-kronecker-recursion} could conceivably be useful for studying 
Kronecker products $s_{\lambda^{(1)}} \ast s_{\lambda^{(2)}}$ for partitions other than hooks.

Although we will not need it, for completeness we record the companion Kronecker product recursion involving $e_j^{\perp}$:
\begin{equation}
\label{ek-perp-equation}
    e_j^{\perp}(s_{\lambda^{(1)}} \ast s_{\lambda^{(2)}}) =
    \sum_{\substack{\mu \vdash j \\ \nu^{(1)}, \nu^{(2)} \vdash n-j}}
    c_{\mu',\nu^{(1)}}^{\lambda^{(1)}} \cdot c_{\mu,\nu^{(2)}}^{\lambda^{(2)}}
    (s_{\nu^{(1)}} \ast s_{\nu^{(2)}}).
\end{equation}
Equation~\eqref{ek-perp-equation} has the same right-hand side as Equation~\eqref{hk-perp-equation},
except that in the first Littlewood-Richardson
coefficient the partition $\mu \vdash k$ is replaced by its conjugate $\mu' \vdash k$.
Equation~\eqref{ek-perp-equation} may be proven in the same way as Equation~\eqref{hk-perp-equation}, except one
uses the formula 
\begin{equation}
    \dim(\Hom_{\symm_j}(\mathrm{sign}_{\symm_j}, S^{\mu^{(1)}} \otimes S^{\mu^{(2)}})) = \begin{cases}
    1 & \mu^{(1)} = (\mu^{(2)})', \\
    0 & \text{otherwise}
    \end{cases}
\end{equation}
for the multiplicity of the sign representation in a Kroenecker product of irreducibles.

Our application of Lemma~\ref{schur-kronecker-recursion} may be stated as follows.
In order to motivate the next result and understand its proof, it will be useful to recall that a product 
$e_a h_b$ of an elementary symmetric function with a homogeneous symmetric function is a sum 
\begin{equation}
    \label{eh-to-hook}
    e_a h_b = s_{(b,1^a)} + s_{(b+1,1^{a-1})} 
\end{equation}
of two successive hook-shaped Schur functions.

\begin{lemma}
\label{hook-kronecker-recursion}
For any $j \geq 1$ and any integers $k, \ell,$ and $m$ with $k + \ell + m = n$ we have
\begin{multline}
    h_j^{\perp} \left(s_{(k+\ell,1^m)} \ast s_{(k+m,1^{\ell})} - s_{(k+\ell+1,1^{m-1}} \ast s_{(k+m+1,1^{\ell-1})} \right) \\
    = h_{k + \ell - j} e_m \ast h_{k+m-j} e_{\ell} - h_{k+\ell}e_{m-j} \ast h_{k+m}e_{\ell-j}.
\end{multline}
\end{lemma}

\begin{proof}
By telescoping sums, we may (and will) prove the equivalent assertion
\begin{equation}
\label{reformulation-one}
    h_j^{\perp}(s_{(k+\ell,1^m)} \ast s_{(k+m,1^{\ell})}) = 
    \sum_{r = 0}^j h_{k + \ell - j + r} e_{m-r} \ast h_{k + m - j+ r} e_{\ell - r}.
\end{equation}
Applying Equation~\eqref{eh-to-hook}, the right-hand side of Equation~\eqref{reformulation-one}
reads
\begin{multline}
    \label{equation-two}
    \sum_{r = 0}^j h_{k + \ell - j + r} e_{m-r} \ast h_{k + m - j+ r} e_{\ell - r} \\= 
    \sum_{r = 0}^j \left(s_{(k+\ell-j+r,1^{m-r})} + s_{(k+\ell-j+r+1,1^{m-r-1})}  \right) \ast
    \left( s_{(k+m-j+r,1^{\ell-r})} + s_{(k+m-j+r+1,1^{\ell-r-1})} \right).
\end{multline}
As for the left-hand side of Equation~\eqref{reformulation-one},
Lemma~\ref{schur-kronecker-recursion} yields
\begin{equation}
    \label{equation-three}
    h_j^{\perp} \left( s_{(k+\ell,1^m)} \ast s_{(k+m,1^{\ell})}  \right) = 
    \sum_{\substack{\mu \vdash j \\ \nu^{(1)}, \nu^{(2)} \vdash n-j}}
    c_{\mu,\nu^{(1)}}^{(k+\ell,1^m)} \cdot c_{\mu,\nu^{(2)}}^{(k+m,1^{\ell})} 
    (s_{\nu^{(1)}} \ast s_{\nu^{(2)}}).
\end{equation}
The Littlewood-Richardson coefficient $c_{\mu,\nu^{(1)}}^{(k+\ell,1^m)}$ is nonzero only when 
both $\mu$ and $\nu^{(1)}$ are hooks.
In this case, we have $\mu = (j-r,1^r)$ for some $r \in \{0,1,\dots,j-1\}$ and
the Littlewood-Richardson Rule implies
\begin{equation}
    \label{equation-four}
    s_{(j-r,1^r)} s_{(c,1^{n-j-c})} = s_{(j-r+c,1^{r+n-j-c})} +  
    s_{(j-r+c-1,1^{r+n-j-c+1})} + R
\end{equation}
where $R$ is a sum of Schur functions indexed by non-hook partitions.
Therefore, the Schur function $s_{(k+\ell,1^m)}$ appears in this product 
precisely when $j-r+c = k+\ell$ or $j-r+c-1 = k + \ell$ or, in other words,
\begin{center}
    $c = k + \ell - j + r \quad \quad$ or $\quad \quad c = k+\ell-j+r+1.$
\end{center}
Similarly, the Schur function $s_{(k+m,1^{\ell})}$ appears in the product
$s_{(j-r,1^r)} s_{(c,1^{n-j-c})}$ precisely when
\begin{center}
    $c = k + m - j + r \quad \quad$ or $\quad \quad c = k+m-j+r+1.$
\end{center}
This means that 
\begin{align}
    \sum_{\substack{\mu \vdash j \\ \nu^{(1)}, \nu^{(2)} \vdash n-j}}
    &c_{\mu,\nu^{(1)}}^{(k+\ell,1^m)} \cdot c_{\mu,\nu^{(2)}}^{(k+m,1^{\ell})} 
    (s_{\nu^{(1)}} \ast s_{\nu^{(2)}}) \\
    =& \sum_{r=0}^{j-1} \sum_{c,d} c_{(j-r,1^r),(c,1^{n-j-c})}^{(k+\ell,1^m)} \cdot
    c_{(j-r,1^r),(d,1^{n-j-d})}^{(k+m,1^{\ell})}
    \left( s_{(c,1^{n-j-c})} \ast s_{(d,1^{n-j-d})} \right) \\
    =& \sum_{r = 0}^j \left(s_{(k+\ell-j+r,1^{m-r})} + s_{(k+\ell-j+r+1,1^{m-r-1})}  \right) \ast 
    \left( s_{(k+m-j+r,1^{\ell-r})} + s_{(k+m-j+r+1,1^{\ell-r-1})} \right).
\end{align}
and applying Equation~\eqref{equation-two} finishes the proof.
\end{proof}

We turn to the Theta operator side of Theorem~\ref{main-theorem}. Our starting point is an identity of 
D'Adderio and Romero \cite[Theorem 8.2]{DR}.  
We use the standard $q$-analogs
\begin{equation}
    [n]_q \coloneqq 1 + q + \cdots + q^{n-1} \quad \quad [n]!_q \coloneqq [n]_q [n-1]_q \cdots [1]_q \quad \quad 
    {n \brack k}_q \coloneqq \frac{[n]!_q}{[k]!_q \cdot [n-k]!_q}
\end{equation}
of numbers, factorials, and binomial coefficients together with the convention
${n \brack k}_q = 0$ whenever $k < 0$ or $k > n$.
We will also need the symmetric functions $E_{n,k} = E_{n,k}(\xx;q)$ introduced in \cite{GH}
which may be defined plethystically by
\begin{equation}
E_{n,k} \coloneqq q^k \sum_{r=0}^k q^{{r \choose 2}} {k \brack r}_q (-1)^r 
e_n\left[  
\xx \frac{1-q^{-r}}{1-q}
\right].
\end{equation}

\begin{theorem}
\label{first-theta-recursion}
{\em (D'Adderio-Romero \cite{DR})}
For any integers $j, m, \ell,$ and $k$ we have
\begin{multline*}
    h_j^{\perp} \Theta_m \Theta_{\ell} \widetilde{H}_k(\xx;q,t) = 
    \sum_{r = 0}^j {k \brack r}_q \sum_{a = 0}^k \sum_{b = 1}^{j-r+a} 
    \Theta_{m-j+r} \Theta_{\ell + k - j - a} \nabla E_{j-r+a,b} \\
    \times \left( 
    q^{{k-r-a \choose 2}} {b-1 \brack a}_q {b+r-a-1 \brack k-a-1}_q +
    q^{{k-r-a+1 \choose 2}} {b-1 \brack a-1}_q {b+k-r-a \brack k-a}_q
    \right).
\end{multline*}
\end{theorem}

We will use Theorem~\ref{first-theta-recursion} to obtain a recursion for the Theta operators at $q=t=0$. In order to do this, we must first replace 
the left-hand side of Theorem~\ref{first-theta-recursion}
with something closer to the expression
$\Theta_i \Theta_j \nabla e_n \mid_{q=t=0}$.
This is accomplished with the following lemma.

\begin{lemma}
\label{nabla-Hk-t=0-lemma}
For any integers $m, \ell,$ and $k$ we have
\begin{equation}
    \Theta_m \Theta_\ell \nabla e_k \mid_{t=0} =  \Theta_m \Theta_\ell \widetilde{H}_k(\xx;q,t) \mid_{t=0}. \label{nabla-Hk-t=0}
\end{equation}
\end{lemma}

\begin{proof}
The first step is to note that 
\begin{equation}
\label{ht-to-nabla}
    \widetilde{H}_k(\xx;q,t) = \nabla e_k \mid_{t=0}.
\end{equation}
The left side of Equation~\eqref{ht-to-nabla}
is well-known to be the graded Frobenius image of the coinvariants of the symmetric group with a single set of bosonic variables. The right side is the bigraded Frobenius image of the coinvariants with two sets of bosonic variables, with the x-variables set to zero. Therefore, the two sides are the same.

We claim that for any symmetric function $G$, we have
\begin{equation}
\label{G-at-zero}
    \left(\Theta_{e_\lambda} G \right) \mid_{t=0} =  \left(\Theta_{e_\lambda} \left( G \mid_{t=0} \right) \right) \mid_{t=0}.
\end{equation}
It is sufficient to verify
\eqref{G-at-zero} over a basis. If $G = \widetilde{H}_\mu(\xx;q,t)$, then $G\mid_{t=0} = \widetilde{H}_\mu(\xx;q,0)$ is a modified Hall-Littlewood symmetric function; these also give a basis for symmetric functions. Let $D_{\mu,\nu}(q,t)$ be the coefficients in the expansion
\begin{equation}
\label{D-expansion}
    \widetilde{H}_\mu(\xx;q,0) = \sum_{\nu} D_{\mu,\nu}(q,t) \widetilde{H}_\nu(\xx;q,t)
\end{equation}
Setting $t=0$, we find that
\begin{equation}
    D_{\mu,\nu}(q,0) = \delta_{\nu, \mu}.
\end{equation}
Now, applying Theta operators to both sides of \eqref{D-expansion}, we find
\begin{equation}
\label{Theta-D-expansion}
    \Theta_{e_\lambda}\widetilde{H}_\mu(\xx;q,0) = \sum_{\lambda} D_{\mu,\nu}(q,t) \Theta_{e_\lambda} \widetilde{H}_\nu(\xx;q,t).
\end{equation}

Assuming that $\Theta_{e_\lambda} \widetilde{H}_\nu(\xx;q,t)$ has no poles at $t=0$, then setting $t=0$ on both sides of \eqref{Theta-D-expansion} we get
\begin{align}
    \Theta_{e_\lambda}\widetilde{H}_\mu(\xx;q,0) \mid_{t=0} & = \sum_{\lambda} D_{\mu,\nu}(q,0) \Theta_{e_\lambda} \widetilde{H}_\nu(\xx;q,t) \mid_{t=0} \\
    & = \Theta_{e_\lambda} \widetilde{H}_\mu(\xx;q,t) \mid_{t=0},
\end{align}
which proves \eqref{G-at-zero}. To show that there are no poles at $t=0$, we see that by definition of Theta operators,
\begin{equation}
    \Theta_{e_\lambda} \widetilde{H}_\mu(\xx;q,t) = \sum_{\rho} d^{\lambda}_{\mu,\rho}(q,t) \frac{\Pi_\rho}{\Pi_\mu} \widetilde{H}_\rho(\xx;q,t),
\end{equation}
where $d^{\lambda}_{\mu,\rho}(q,t)$ is the coefficient of $\widetilde{H}_\rho(\xx;q,t)$ in 
$e_\lambda[\xx/M]\widetilde{H}_\mu(\xx;q,t) $. At $t=0$, this product gives
\begin{align}
e_\lambda\left[ \frac{\xx}{M} \right]\widetilde{H}_\mu(\xx;q,t) \mid_{t=0}
& = e_\lambda\left[ \frac{\xx}{1-q} \right]\widetilde{H}_\mu(\xx;q,0)  \\ &=
\sum_{\rho} C_{\mu,\rho}(q) \widetilde{H}_\rho(\xx;q,0) 
\end{align}
for some coefficients $C_{\mu,\rho}(q)$. Since the modified Hall-Littlewood symmetric functions are a basis, we must then have 
$ d^{\lambda}_{\mu,\rho}(q,0) = C_{\mu,\rho}(q)$. This means $d^{\lambda}_{\mu,\rho}(q,t)$ has no pole at $t=0$, and since $\Pi_\rho/\Pi_\mu$ also has no pole at $t=0$, we can conclude that \eqref{G-at-zero} holds
for any symmetric function $G$, proving the lemma.
\end{proof}

With Lemma~\ref{nabla-Hk-t=0-lemma} in hand, we apply 
Theorem~\ref{first-theta-recursion}
to get a recursive expression for the action of $h_j^{\perp}$
on the symmetric functions
$\Theta_m \Theta_{\ell} \nabla e_k \mid_{q = t = 0}$ of 
Theorem~\ref{main-theorem}.

\begin{lemma}
For any integers $j,m,\ell,$ and $k$ we have
\begin{align*}
    h_j^{\perp} \Theta_m \Theta_{\ell} \nabla e_k \mid_{q = t = 0} = 
    &\sum_{r = 0}^j \Theta_{m-r} \Theta_{\ell - r} \nabla e_{k-j+2r} \mid_{q = t = 0} \\
    &+ \sum_{r = 0}^{j-1} (\Theta_{m-r-1} \Theta_{\ell-r} + \Theta_{m-r} \Theta_{\ell-r-1}) 
    \nabla e_{k-j+2r+1} \mid_{q=t=0} \\
    &+ \sum_{r=0}^{j-2} \Theta_{m-r-1} \Theta_{\ell-r-1} \nabla e_{k-j+2r+2} \mid_{q=t=0}.
\end{align*}
\end{lemma}

\begin{proof}
For convenience, we apply the transformation $a \mapsto k-a$ in Theorem~\ref{first-theta-recursion} 
to obtain
\begin{multline}
\label{second-theta-recursion}
    h_j^{\perp} \Theta_m \Theta_{\ell} \widetilde{H}_k(\xx;q,t) = 
    \sum_{r = 0}^j {k \brack r}_q \sum_{a = 0}^k \sum_{b = 1}^{k+j-r-a} 
    \Theta_{m-j+r} \Theta_{\ell - j + a} \nabla E_{k+j-r-a,b} \\
    \times \left( 
    q^{{a-r \choose 2}} {b-1 \brack k-a}_q {b-k+r+a-1 \brack a-1}_q +
    q^{{r-a \choose 2}} {b-1 \brack k-a-1}_q {b+k-r-a \brack a}_q
    \right)
\end{multline}
Our goal is to evaluate Equation~\eqref{second-theta-recursion} at $q \rightarrow 0$. From Equation~\eqref{nabla-Hk-t=0}, we see the right-hand sides of both Equation~\eqref{second-theta-recursion} and the statment in the Lemma agree.

At $q \rightarrow 0$, in order for the expression
in the parentheses of \eqref{second-theta-recursion}
to be nonzero, we must have
$-1 \leq a-r \leq 1$ (for otherwise both summands involve only positive powers of $q$).
Moreover, by examining the $q$-binomials, we see that $r \leq k$
(from the $q$-binomial on the first line of \eqref{second-theta-recursion})
and $a \leq j$ (since $b \leq k+j-r-a$, or $b-k+r+a \leq j$, the rightmost 
$q$-binomial in both summands vanishes unless $a \leq j$). Notice that, when the summands in the parentheses evaluate at $q \rightarrow 0$
to something nonzero, once evaluated they do not depend on $b$ anymore, so we can isolate the sum over $b$ and use
the identity
\begin{equation}
    \sum_{b = 1}^{k+j-r-a} \nabla E_{k+j-r-a} = \nabla e_{k+j-r-a}.
\end{equation}
We have three separate cases depending on the value of $a-r \in \{-1,0,1\}$.

{\bf Case 1.} $a-r = -1$.

In this case, the left summand $q^{{a-r \choose 2}} {b-1 \brack k-a}_q {b-k+r+a-1 \brack a-1}_q$ in the parentheses
vanishes and the right summand 
$q^{{r-a \choose 2}} {b-1 \brack k-a-1}_q {b+k-r-a \brack a}_q$ evaluates to 1 at $q \rightarrow 0$.
Since $a \geq 0$, the sum restricts to $r \geq 1$ and we have a contribution of
\[ \sum_{r=1}^j \Theta_{m-j+r} \Theta_{\ell-j+r-1} \nabla e_{k+j-2r+1} \mid_{q=t=0} ; \] now making the change of variables $r \mapsto j-r$ we get
\begin{equation}
\label{case-one-contribution}
    \sum_{r=0}^{j-1} \Theta_{m-r} \Theta_{\ell-r-1} \nabla e_{k-j+2r+1} \mid_{q=t=0}.
\end{equation}

{\bf Case 2.} $a-r = 0$.

In this case, both summands in the parentheses evaluate to 1 at $q \rightarrow 0$, except when $a = r =0 $
(in which case the left summand evaluates to 0) or when $a=r=j$
(in which case the right summand evaluates to 0). The contribution of this case is 

\[ \sum_{r = 0}^j (2 -\delta_{r,0} - \delta_{r,j}) \Theta_{m-j+r} \Theta_{\ell-j+r} \nabla e_{k+j-2r} \mid_{q=t=0} \] where $\delta_{r,0}$ and $\delta_{r,j}$ are Kronecker deltas. We can rewrite this as \[ \sum_{r = 0}^j \Theta_{m-j+r} \Theta_{\ell-j+r} \nabla e_{k+j-2r} \mid_{q=t=0} + \sum_{r = 1}^{j-1}  \Theta_{m-j+r} \Theta_{\ell-j+r} \nabla e_{k+j-2r} \mid_{q=t=0} \] and now making the change of variables $r \mapsto j-r$ in the left summand and $r \mapsto j-r-1$ in the right summand we get
\begin{equation}
\label{case-two-contribution}
    \sum_{r = 0}^{j} \Theta_{m-r} \Theta_{\ell-r} \nabla e_{k-j+2r} \mid_{q=t=0} + \sum_{r = 0}^{j-2}  \Theta_{m-r-1} \Theta_{\ell-r-1} \nabla e_{k-j+2r+2} \mid_{q=t=0}.
\end{equation}

{\bf Case 3.} $a-r = 1$.

In this case the right summand in the parentheses vanishes at $q \rightarrow 0$ and the left summand evaluates to 1.
Since $a \leq j$, the sum restricts to $r \leq j-1$. We get a contribution of \[ \sum_{r=0}^{j-1} \Theta_{m-j+r} \Theta_{\ell-j+r+1} \nabla e_{k+j-2r-1} \mid_{q=t=0}; \]
now making the change of variables $r \mapsto j-r-1$ we get
\begin{equation}
\label{case-three-contribution}
    \sum_{r=0}^{j-1} \Theta_{m-r-1} \Theta_{\ell-r} \nabla e_{k-j+2r+1} \mid_{q=t=0}.
\end{equation}

The lemma follows immediately by combining the contributions \eqref{case-one-contribution}, \eqref{case-two-contribution},
and \eqref{case-three-contribution}.
\end{proof}

We have all the tools we need to prove our main result Theorem~\ref{main-theorem}.

\begin{proof}
 (of Theorem~\ref{main-theorem})
 By Theorem~\ref{fdr-isomorphism-type}, we aim to show 
 \begin{equation}
 \label{final-goal-equation}
 \Theta_m \Theta_{\ell} \nabla e_k \mid_{q=t=0} = s_{(k+\ell,1^m)} \ast s_{(k+m,1^{\ell})} - 
 s_{(k+\ell+1,1^{m-1})} \ast s_{(k+m+1,1^{\ell-1})}
 \end{equation}
 for all integers $k,\ell,m \geq 0$.
 Notice that both sides of Equation~\eqref{final-goal-equation} are symmetric functions of 
 degree $m+\ell+k$.  When $m+\ell+k = 0$ both sides specialize to 1 so we assume that 
 $m+\ell+k > 0$ and both sides have positive total degree.
 
 We prove this result by induction on the total degree $m+\ell+k$. Let $j \geq 1$.
 Using Lemma~\ref{second-theta-recursion},
 we apply the operator $h_j^{\perp}$ to the left-hand side of  
 Equation~\eqref{final-goal-equation} yielding
   \begin{align*}
      h_j^\perp \Theta_{m} \Theta_{\ell} & \nabla e_k \rvert_{q=t=0} = \sum_{r=0}^j \left. \Theta_{m-r} \Theta_{\ell-r} \nabla e_{k-j+2r} \right\rvert_{q=t=0} \\
      & \quad + \sum_{r=0}^{j-1} \left. \left( \Theta_{m-r-1} \Theta_{\ell-r} + \Theta_{m-r} \Theta_{\ell-r-1} \right) \nabla e_{k-j+2r+1} \right\rvert_{q=t=0} \\
      & \quad + \sum_{r=0}^{j-2} \left. \Theta_{m-r-1} \Theta_{\ell-r-1} \nabla e_{k-j+2r+2} \right\rvert_{q=t=0} \\
& \\
      \text{(ind. hp.)} & = \sum_{r=0}^j  s_{k+\ell-j+r, 1^{m-r}} \ast s_{k+m-j+r, 1^{\ell-r}} - s_{k+\ell-j+r+1, 1^{m-r-1}} \ast s_{k+m-j+r+1, 1^{\ell-r-1}} \\
      & \quad + \sum_{r=0}^{j-1} s_{k+\ell-j+r+1, 1^{m-r-1}} \ast s_{k+m-j+r, 1^{\ell-r}} - s_{k+\ell-j+r+2, 1^{m-r-2}} \ast s_{k+m-j+r+1, 1^{\ell-r-1}} \\
      & \quad + \sum_{r=0}^{j-1} s_{k+\ell-j+r, 1^{m-r}} \ast s_{k+m-j+r+1, 1^{\ell-r-1}} - s_{k+\ell-j+r+1, 1^{m-r-1}} \ast s_{k+m-j+r+2, 1^{\ell-r-2}} \\
      & \quad + \sum_{r=0}^{j-2} s_{k+\ell-j+r+1, 1^{m-r-1}} \ast s_{k+m-j+r+1, 1^{\ell-r-1}} - s_{k+\ell-j+r+2, 1^{m-r-2}} \ast s_{k+m-j+r+2, 1^{\ell-r-2}} \\
& \\
      (1)-(2) \quad & = \sum_{r=0}^j  s_{k+\ell-j+r, 1^{m-r}} \ast s_{k+m-j+r, 1^{\ell-r}} - \sum_{r=0}^j s_{k+\ell-j+r+1, 1^{m-r-1}} \ast s_{k+m-j+r+1, 1^{\ell-r-1}} \\
      (3)-(4) \quad & \quad + \sum_{r=0}^{j-1} s_{k+\ell-j+r+1, 1^{m-r-1}} \ast s_{k+m-j+r, 1^{\ell-r}} - \sum_{r=1}^{j} s_{k+\ell-j+r+1, 1^{m-r-1}} \ast s_{k+m-j+r, 1^{\ell-r}} \\
      (5)-(6) \quad & \quad + \sum_{r=0}^{j-1} s_{k+\ell-j+r, 1^{m-r}} \ast s_{k+m-j+r+1, 1^{\ell-r-1}} - \sum_{r=1}^{j} s_{k+\ell-j+r, 1^{m-r}} \ast s_{k+m-j+r+1, 1^{\ell-r-1}} \\
      (7)-(8) \quad & \quad + \sum_{r=0}^{j-2} s_{k+\ell-j+r+1, 1^{m-r-1}} \ast s_{k+m-j+r+1, 1^{\ell-r-1}} - \sum_{r=2}^{j} s_{k+\ell-j+r, 1^{m-r}} \ast s_{k+m-j+r, 1^{\ell-r}} 
      \end{align*}
      where the second equality used induction on degree and the numerals 
      $(1), \dots, (8)$ on the left of the last expression abbreviate the 
      eight sums therein. We rearrange and make cancellations
      in these sums, obtaining
%
%
\begin{align*}
      (1)-(8) \quad & = s_{k+\ell-j, 1^{m}} \ast s_{k+m-j, 1^{\ell}} + s_{k+\ell-j+1, 1^{m-1}} \ast s_{k+m-j+1, 1^{\ell-1}} \\
      (3)-(4) \quad & \quad + s_{k+\ell-j+1, 1^{m-1}} \ast s_{k+m-j, 1^{\ell}} - s_{k+\ell+1, 1^{m-j-1}} \ast s_{k+m, 1^{\ell-j}} \\
      (5)-(6) \quad & \quad + s_{k+\ell-j, 1^{m}} \ast s_{k+m-j+1, 1^{\ell-1}} - s_{k+\ell, 1^{m-j}} \ast s_{k+m+1, 1^{\ell-j-1}} \\
      (7)-(2) \quad & \quad - s_{k+\ell, 1^{m-j}} \ast s_{k+m, 1^{\ell-j}} - s_{k+\ell+1, 1^{m-j-1}} \ast s_{k+m+1, 1^{\ell-j-1}}. 
\end{align*}

Reindexing these eight summands with $(a), \dots, (h)$
and applying Equation~\eqref{eh-to-hook} gives
\begin{align*}
      (a)+(b) \quad & = s_{k+\ell-j, 1^{m}} \ast s_{k+m-j, 1^{\ell}} + s_{k+\ell-j+1, 1^{m-1}} \ast s_{k+m-j+1, 1^{\ell-1}} \\
      (c)-(d) \quad & \quad + s_{k+\ell-j+1, 1^{m-1}} \ast s_{k+m-j, 1^{\ell}} - s_{k+\ell+1, 1^{m-j-1}} \ast s_{k+m, 1^{\ell-j}} \\
      (e)-(f) \quad & \quad + s_{k+\ell-j, 1^{m}} \ast s_{k+m-j+1, 1^{\ell-1}} - s_{k+\ell, 1^{m-j}} \ast s_{k+m+1, 1^{\ell-j-1}} \\
      -(g)-(h) \quad & \quad - s_{k+\ell, 1^{m-j}} \ast s_{k+m, 1^{\ell-j}} - s_{k+\ell+1, 1^{m-j-1}} \ast s_{k+m+1, 1^{\ell-j-1}} \\
& \\
      (a)+(c) \quad & = h_{k+\ell-j} e_{m} \ast s_{k+m-j, 1^{\ell}} \\
      (b)+(e) \quad & \quad + h_{k+\ell-j} e_{m} \ast s_{k+m-j+1, 1^{\ell-1}} \\
      -(d)-(g) \quad & \quad - h_{k+\ell} e_{m-j} \ast s_{k+m, 1^{\ell-j}} \\
      -(f)-(h) \quad & \quad - h_{k+\ell} e_{m-j} \ast s_{k+m+1, 1^{\ell-j-1}} \\
& \\
      & = h_{k+\ell-j} e_{m} \ast h_{k+m-j} e_{\ell} - h_{k+\ell} e_{m-j} \ast h_{k+m} e_{\ell-j} \\
& \\
%
    & = h_j^\perp (s_{k+\ell, 1^m} \ast s_{k+m, 1^\ell} - s_{k+\ell+1, 1^{m-1}} \ast s_{k+m+1, 1^{\ell-1}})
  \end{align*}
  where the last step uses Lemma~\ref{hook-kronecker-recursion}.
  In summary, we have
  \begin{equation}
            h_j^\perp \Theta_{m} \Theta_{\ell}  \nabla e_k \rvert_{q=t=0} =  
            h_j^\perp (s_{k+\ell, 1^m} \ast s_{k+m, 1^\ell} - s_{k+\ell+1, 1^{m-1}} \ast s_{k+m+1, 1^{\ell-1}})
  \end{equation}
  and since this holds for every $j \geq 1$, 
  by Lemma~\ref{h-perp-recursion} we can deduce that Equation~\eqref{final-goal-equation} holds. This completes the proof of Theorem~\ref{main-theorem}.
\end{proof}

\section{Acknowledgements}

B. Rhoades was partially supported by NSF DMS-1953781.
M. Romero was partially supported by the NSF Mathematical Sciences Postdoctoral Research Fellowship DMS-1902731.

\end{document}